\documentclass[reqno,a4paper]{amsart}

\usepackage[mathscr]{eucal}
\usepackage{amssymb,verbatim,graphicx,epstopdf,soul}
\usepackage{enumerate}
\usepackage{graphpap}
\usepackage{mathrsfs}

\newtheorem{thm}{Theorem}[section]

\newtheorem{lem}[thm]{Lemma}

\newtheorem{claim}{Claim}
\newtheorem{indcase}{Induction Case}

\theoremstyle{definition}

\newtheorem{problem}[thm]{Problem}

\newtheorem{remark}[thm]{Remark}

\newcommand{\up}{\textup}

\newcommand{\con}{\operatorname{con}}
\newcommand{\occ}{\operatorname{occ}}

\newcommand{\ba}{\mathbf{a}}
\newcommand{\bb}{\mathbf{b}}

\newcommand{\bc}{\mathbf{c}}

\newcommand{\bM}{\mathbf{M}}

\newcommand{\bp}{\mathbf{p}}
\newcommand{\bq}{\mathbf{q}}

\newcommand{\bu}{\mathbf{u}}
\newcommand{\bv}{\mathbf{v}}
\newcommand{\bw}{\mathbf{w}}
\newcommand{\bx}{\mathbf{x}}
\newcommand{\by}{\mathbf{y}}
\newcommand{\bz}{\mathbf{z}}

\newcommand{\sW}{\mathscr{W}}
\newcommand{\sV}{\mathscr{V}}
\newcommand{\sU}{\mathscr{U}}
\newcommand{\sA}{\mathscr{A}}

\newcommand{\cK}{\mathcal{K}}
\newcommand{\cV}{\mathcal{V}}

\newcommand{\irr}{ir\-re\-dun\-dant}
\newcommand{\nfb}{non-finitely based}

\begin{document}
\title[Irredundant axiomatisability]{Infinite irredundant equational axiomatisability for a finite monoid}
\begin{abstract}
It is shown that a finite monoid can have an infinite irredundant basis of equations.
\end{abstract}
\author[M. Jackson]{Marcel Jackson}
\address{Department of Mathematics and Statistics, La Trobe University VIC 3086, Australia}
\email{m.g.jackson@latrobe.edu.au}
\thanks{The author was supported by ARC Discovery Project DP1094578 and Future Fellowship FT120100666}
\maketitle


\section{Introduction} 
The problem of redundancy in axiomatic systems goes back at least to the millennia-long saga of the possible redundancy of the parallel postulate amongst the axioms for Euclidean geometry.  In the case of axiomatic systems in an algebraic setting, there are also famously difficult problems relating to minimal bases for Boolean algebras \cite{mcc}, and the seminal work of Tarski on irredundant systems~\cite{tar75}.  When a theory admits a finite axiomatisation, then there always exists an axiomatisation without redundancies---an \emph{irredundant axiomatisation}.  Indeed, one can simply remove axioms one by one from any finite axiomatisation until no further redundancies exist.  For systems admitting no finite axiomatisation however, it is possible that every complete axiomatisation has redundancies.  Examples of equational systems without irredundant bases, as well as infinite irredundant equational systems go back to the 1960s (see Tarski \cite{tar66} for example).  Even a single \emph{finite} algebraic structure can be without a finite axiomatisation for its equational theory (Lyndon \cite{lyn}), and in this particular case, not only is it still possible to have no irredundant axiomatisation~\cite{L15,gM}, it appears reasonably unusual for an irredundant axiomatisation to exist~\cite[Problem 2.6]{V}.  

The first examples of finite algebras whose equational theory has an infinite irredundant basis appear to be the semigroups of Jackson~\cite{J05a}. %
More examples have recently been discovered: %
\begin{enumerate}[1]
\item[$\bullet$] by applying the constructions of Jackson~\cite[\S5 and~\S7.1]{J08} to a certain digraph with an infinite {\irr} axio\-ma\-ti\-za\-tion for its universal Horn sentences~\cite[Example 17]{JT}, a type $\langle 2,2,2,1,1\rangle$ algebra of order nine with an infinite {\irr} equational basis is obtained; %

\item[$\bullet$] for each $n \geq 8$, there exists at least one involution semigroup of order $n$ with infinite {\irr} equational basis~\cite{L}. %
\end{enumerate}

In the present article, we consider semigroups with identity element---``\textit{monoids}''---as algebras in the signature $\{\,\cdot\,,1\}$ of type $\langle 2,0\rangle$. %
The case of monoids appears to present particular challenges to infinite irredundant axiomatisation, as we now explain.  

First, fix some notion of \emph{length} of an equation: the length as a word in prefix notation for example (but any reasonable measure of length will suffice).
The non-finite basis property can be broadly described as the property that short equations are insufficient to derive long equations. %
In particular, any basis for a {\nfb} algebra must contain arbitrarily long equations. %
To have an {\irr} basis, it is additionally required that long equations also do not imply too many relatively short equations: otherwise some individual equation can always be removed from any basis, as it will follow from some sufficiently long equations in the basis. %
But in finite algebras, long equations always imply some short equations, by way of identification of variables, and this makes it difficult to avoid redundancies. %
This already goes some way to suggesting why infinite irredundant axiomatisability appears to be a rare property for finite algebras.
In the case of monoids however, ``identification of variables'' may be supplanted by first assigning~$1$ to variables, then followed by eliminating all excessive occurrences of~$1$ with the equations $x1\approx x\approx 1x$. %
This completely simplistic re\-duc\-tion gives long equations enormous deductive influence over short equations under the monoid signature. %
A finite nilpotent monoid in Jackson~\cite{J05a} specifically demonstrates that the property of having an infinite {\irr} axio\-ma\-ti\-za\-tion can be destroyed under a transition from the semigroup signature to the monoid signature. %
This naturally leads to the following problem. %

\begin{problem}[{\cite[page~421]{J05a}}] \label{problem}
Is there a finite monoid with an infinite {\irr} basis of monoid identities? %
\end{problem}

\noindent This problem has also been posed in the updated version of Volkov's survey~\cite{V}, as accessed at the time of writing. %

In the present article, a finite nilpotent monoid with an infinite {\irr} basis of monoid equations is exhibited, thereby providing a positive solution to Problem~\ref{problem}. %
The argument is substantially more involved than the broad approach that was given in Jackson~\cite{J05a}. %
In the monoid signature, the property appears to be extremely fragile, and it seems unlikely to hold anywhere near as commonly as when the examples are considered in the semigroup signature. %

Recall that a \emph{variety} of algebras is an equationally defined class of algebras of the same signature. The present work arose out of other recent developments concerning finite monoids whose variety has continuum many subvarieties.   The connection is as follows: if $\mathcal{V}$ is a variety with an infinite irredundant axiomatisation for its equations and $\mathcal{W}$ is a finitely based variety with $\mathcal{V}\subseteq \mathcal{W}$, then $\mathcal{W}$ has continuum many subvarieties.  The example in the present article  is contained in many finitely based varieties.

\section{Preliminaries}

Standard notation is followed; boldface lower case letters $\ba, \bb, \bc, \ldots, \bz$, often with subscripts and primes, denote words, and non-boldface lower case letters denote variables occurring in words. %
The expression $\ba \leq \bb$ indicates that~$\ba$ is a \emph{subword} or \textit{factor} of~$\bb$, that is, $\bb = \bx \ba \by$ for some possibly empty words~$\bx$ and~$\by$. %

In general, words are taken from some fixed countably infinite alphabet~$\sA$ of variables built over the symbols in $\{a,b,c,\ldots,z\}$, sometimes with subscripts. %
The set of all words in the alphabet~$\sA$ is denoted by~$\sA^*$, which includes the empty word denoted by~$1$, and carries the structure of a monoid with respect to the operation of concatenation. %
As usual, we use~$\sA^+$ to denote $\sA^*\backslash \{ 1\}$. %
The \textit{content} of a word~$\bw$, denoted $\con(\bw)$, is the set of all variables from~$\sA$ that appear in~$\bw$. %
The number of occurrences of a variable~$x$ in a word~$\bw$ is denoted by $\occ(x,\bw)$, so that $\con(\bw)= \{ x\in \sA \,|\, \occ(x,\bw) > 0 \}$. %
A word~$\bw$ is \textit{$n$-limited} if $\occ(x,\bw) \leq n$ for all $x \in \sA$, and a variable~$x$ is \textit{$n$-occurring} in~$\bw$ if $\occ(x,\bw)=n$. %
A variable that is $1$-occurring is said to be \textit{linear} and a word is said to be \emph{linear} if it is $1$-limited. %
%
Reference to specific occurrences of a variable~$x$ in a word~$\bw$ is often required. %
For this purpose, if $\occ(x,\bw)\geq i$, then let~$_ix$ denote the~$i$th occurrence of~$x$ in~$\bw$ from the left. %
We extend this to sequences of occurrences, so that we can say that $_2x\, {}_1z$ is a factor of $xy\underline{xz}zy$ (the pattern of occurrences is underlined), while $_1z\, {}_2y$ is not, even if $zy$ is actually a factor of $xyxzzy$.

An \textit{identity} is an expression $\bw \approx \bw'$ where~$\bw,\bw' \in \sA^*$; we use the phrase \emph{equation} synonymously with identity. %
A monoid~$\bM$ \textit{satisfies} an identity $\bw\approx \bw'$, written $\bM\models \bw\approx \bw'$, if for every substitution $\theta: \sA \to \bM$, the equality $\bw\theta=\bw'\theta$ holds in~$\bM$. %
For any class~$\cK$ of monoids, write $\cK \models \bw \approx \bw'$ to indicate that every $\bM \in \cK$ satisfies $\bw \approx \bw'$. %
The class of all monoids satisfying the identities true on the class $\cK$ is the \emph{variety} generated by $\cK$, which we denote $\mathcal{V}(\cK)$.  By a classical result of Birkhoff, $\mathcal{V}(\cK)$ is also the class of monoids obtained from $\cK$ by taking direct products, submonoids and monoid homomorphisms.  The class of all monoids satisfying a system of monoid identities $\Sigma$ is also a variety, which we denote by $[\![\Sigma]\!]$; note that $\mathcal{V}([\![\Sigma]\!])=[\![\Sigma]\!]$.  A variety is said to be \emph{finitely generated} if it is equal to $\mathcal{V}(\bM)$ for a single finite monoid $\bM$.

For any set~$\Sigma$ of monoid identities, write $\Sigma \models \bw \approx \bw'$ to mean that the class of all monoids satisfying $\Sigma$ also satisfies $\bw \approx \bw'$. %
We write $\Sigma \vdash \bw \approx \bw'$ to mean that there is a \textit{derivation} of $\bw \approx \bw'$ from $\Sigma$: formally, there is a sequence %
\[
\bw=\bu_0, \bu_1, \ldots, \bu_r = \bw' %
\]
of distinct words such that for each $i \in \{0,\ldots,r-1\}$, there exist an identity $\bp_i\approx \bq_i\in \Sigma$, a substitution $\theta_i: \sA \to \sA^*$, and words $\ba_i, \bb_i \in \sA^*$ such that $\{\bu_i,\bu_{i+1}\}=\{\ba_i(\bp_i\theta_i)\bb_i,\ba_i(\bq_i\theta_i)\bb_i\}$. %
Birkhoff's completeness theorem of equation logic states that  $\Sigma\models\bu\approx \bv$ if and only if $\Sigma\vdash \bu\approx \bv$; see Burris and Sankappanavar~\cite[Theorem~14.19]{BS}. %

An \textit{isoterm} for a monoid variety~$\cV$ (or for any monoid generating~$\cV$, or for any system of identities defining~$\cV$) is a word~$\bw$ such that for any other word~$\bw'$, if $\cV \models \bw \approx \bw'$ then $\bw\approx\bw'$ holds in all monoids. %
In other words, once all instances of~$1$ have been removed by applications of $x1 \mapsto x$ and $1x \mapsto x$, the two words~$\bw$ and~$\bw'$ are reduced to identical expressions. %

For any set $\mathscr{X} \subseteq \con(\bw)$, let $\bw[\mathscr{X}]$ denote the result of applying the substitution that fixes all variables in~$\mathscr{X}$ and assigns the value~$1$ to all variables in $\sA \backslash \mathscr{X}$. %
For example, $\bw[\{x\}]=x^{\occ(x,\bw)}$ for any $x\in \con(\bw)$. %
It is convenient to write $\bw[x_1,\ldots,x_r]$ in place of $\bw[\{x_1,\ldots,x_r\}]$. %

Let $\bw\approx \bw'$ be any identity and let~$x$ and~$y$ be distinct variables. %
Then the un\-or\-dered pair $\{x,y\}$ is said to be \textit{stable} in $\bw \approx \bw'$ if $\bw[x,y] = \bw'[x,y]$, and \textit{unstable} otherwise. %
The ordered pair $(x,y)$ is called a \textit{critical pair} in $\bw \approx \bw'$ if there exist $i \leq \occ(x,\bw)$ and $j \leq\occ(y,\bw)$ such that ${_ix}\, {_jy}$ is a factor of~$\bw$ but in~$\bw'$, the $i$th occurrence of~$x$ occurs after the $j$th occurrence of~$y$. %
In this case we also say that the pair $({_ix},{_jy})$ is a \textit{critical occurrence pair}. %
This concept has most relevance when the only identities of interest are \textit{balanced} in the sense that $\occ(z,\bw)=\occ(z,\bw')$ for all $z \in \sA$. %
It is easily seen that any nontrivial balanced identity has unstable critical pairs and unstable critical occurrence pairs. %
The standard method for demonstrating that a set~$\Sigma$ of identities is a basis for a nilpotent monoid~$\bM$ is to first reduce the problem to that of deducing balanced identities, and then show that~$\Sigma$ may be used to reduce the number of unstable pairs in any balanced identity satisfied by~$\bM$.  Usually this involves showing how to remove some critical pair from the set of unstable pairs.  There are only finitely many unstable pairs in $\bu\approx \bv$, so repeating this procedure starting from $\bu$, we eventually arrive at $\bv$, thereby giving a deduction of $\bu\approx \bv$.  We employ this method in the present article to show that a chosen irredundant system is a basis for a certain nilpotent monoid. %

The notion of stability can be extended to words relative to a fixed monoid. %
An unordered pair $\{x,y\}$ is \textit{stable} in a word~$\bw$ with respect to a monoid~$\bM$ if $\bw[x,y]=\bw'[x,y]$ whenever $\bM \models \bw \approx \bw'$. %
It is routinely seen that~$\bw$ is an isoterm for~$\bM$ if and only if every pair of variables of~$\bw$ is stable in~$\bw$ with respect to~$\bM$. %
Moreover, to prove this, it suffices to establish stability for all pairs of variables having a neighboring occurrence in~$\bw$. %

\subsection{Rees quotients of free monoids} %

For any set $\sW \subseteq \sA^*$ of words, the \textit{factorial closure} of~$\sW$ is the set
\[
\sW^\leq=\{\bu \in \sA^* \,|\, \text{$\bu \leq \bw$ for some $\bw \in \sW$}\}, %
\]
which consists of all words that are subwords of words in $\sW$.
The set $I(\sW) = \sA^* \backslash \sW^\leq$ forms an ideal of $\sA^*$, and the corresponding Rees quotient $\sA^*/I(\sW)$ will be denoted by $\bM(\sW)$.   %
For example, the universe of the monoid $\bM(\{xyx\})$ is the set $\{1,x,y,xy,yx,xyx,0\}$.
Clearly, $\bM(\sW)$ is finite if and only if $\sW$ is finite and
in this case, it is more convenient to write $\bM(\bw_1,\ldots,\bw_r)$ instead of $\bM(\{\bw_1,\ldots,\bw_r\})$. %

\begin{remark}
The Rees quotient $\sA^*/I(\sW)$ is more commonly denoted in the literature by $\mathrm{S}(\sW)$; %
see, for example, Jackson~\cite{J00,J01,J05b}, Jackson and Sapir~\cite{JS}, and Lee~\cite{L09,L11,L12a,L13}. %
However, we feel this notation would be better for the Rees quotient $\sA^+\big/\big(\sW^\leq \,\backslash \{1\}\big)$, and as we here work in the strict monoid signature, we take this opportunity to use a better targeted notation.%
\end{remark}

\begin{lem}[Jackson~{\cite[Lemma~3.3]{J05b}}] \label{lem:isoterm} %
Let $\cV$ be any monoid variety and $\sW \subseteq \sA^*$\up. %
Then $\bM(\sW)\in \cV$ if and only if every word in $\sW$ is an isoterm for~$\cV$\up. %
\end{lem}

\section{A finite monoid with an infinite irredundant identity basis}
Our example is of the form $\bM(\sV)$, where the set $\sV$ is a quite complicated set of~37 words.  The proof takes the following 3 step approach.  
\begin{itemize}
\item[(A)] We present an infinite system $\Sigma$ of identities and show it is {\irr}.  
\item[(B)] We present the set $\sV$, which we show are isoterms for $\Sigma$.  Hence, by Lemma \ref{lem:isoterm} we have $\mathbb{HSP}(\bM(\sV)))\subseteq [\![\Sigma]\!]$.
\item[(C)] We show that $\mathbb{HSP}(\bM(\sV)))= [\![\Sigma]\!]$ by showing that $\Sigma\vdash \bu\approx \bv$ whenever $\bM(\sV)\models \bu\approx \bv$.
\end{itemize}

%
%
\subsection{Step A}
Consider the word
\[
\bw_n:=x_0z_1xyz_2x_1x_0x_2x_1x_3x_2\dots x_{n}x_{n-1}z_1xyz_2x_n.
\]
and let $\bw_n'$ denote the result of switching the position of $x$ and $y$ in this word.    We are going to show that the monoid variety determined by
\[
\Sigma:=\{\bw_n\approx \bw_n'\mid n\geq 2\}\cup\{xt_1xt_2x\approx x^3t_1t_2\approx t_1t_2x^3,x^3\approx x^4\}
\]
 is finitely gen\-er\-ated.  We will make frequent use of the fact that, for $n>1$, \emph{the only subwords of $\bw_n$ that appear more than once in $\bw_n$ are individual letters, and subwords of $z_1xyz_2$}.  This is not true for the word $\bw_1$, and the greater repetition in this word makes the law $\bw_1\approx \bw_1'$ a consequence of those in $\Sigma$: consider the assignment $\theta$ giving all variables $x_i$ the value $1$, the variable $z_1$ the value $x_0z_1$, the variable $z_2$ the value $z_2x_1$, and fixing $x$ and $y$.  Then for any $n>1$ we have $\bw_n\theta=\bw_1$ and $\bw_n'\theta=\bw_1'$.
 \begin{lem}
 The identity system $\Sigma$ is irredundant.
 \end{lem}
 \begin{proof}
 The identities are very similar to a number of systems shown to be irredundant by the author and Lee in \cite{jaclee}, so we give only a sketch of the details.  First, it is easy to see that the identities in $\{xt_1xt_2x\approx x^3t_1t_2t_3\approx t_1t_2x^3,x^3\approx x^4\}$ cannot be applied nontrivially to those in $\{\bw_n\approx \bw_n'\mid n\in\mathbb{N}\}$ and vice versa: in the forward direction we would need to assign $1$ to $x$, rendering the identities trivial, and in the reverse direction, all letters are $2$-occurring, so cannot cover a linear letter such as $t_1$ and $t_2$.  So it suffices to show that for $n\neq m$ there is no nontrivial deduction $\bw_n\approx \bw_n'\vdash \bw_m\approx \bv$ for some $\bv\neq\bw_m$.  Note that $\{x,y\}$ is the only unstable pair in $\bw_n\approx \bw_n'$, so a nontrivial deduction using  $\bw_n\approx \bw_n'$ requires a substitution $\theta$ in which neither $x\theta$ nor $y\theta$ equals $1$.

For any $k\in\mathbb{N}$ let us write $\bw_k$ as $x_0z_1xyz_2\bu_k z_1xyz_2x_k$, where 
\[
\bu_k=x_1x_0x_2x_1x_3\dots x_{k-1} x_{k-2}x_kx_{k-1}.
\]
Assume that $\theta$ is a substitution with $\bw_n\theta\leq \bw_m$ and with $x\theta\neq 1$ and $y\theta\neq 1$.  Now, 
$z_1xyz_2\theta$ is a subword appearing twice in $\bw_n\theta$ and
$|z_1xyz_2\theta|\geq |xy\theta|\geq 2$.  The only subwords of length more than~$1$ that appear twice in $\bw_m$ are subwords of $z_1xyz_2$.  Hence $z_1xyz_2\theta\leq z_1xyz_2$ which then also gives $\bu_m\leq\bu_n\theta$.  We wish to show that $z_1xyz_2\theta= z_1xyz_2$.  Assume for contradiction, that $z_1$ is not contained in $z_1xyz_2\theta$ (a contradiction in the case of $z_2$ will follow by symmetry).  In this situation, the second occurrence of $z_1$ is covered by $\bu_n\theta$; in particular there is $i\leq n$ such that $z_1\leq x_i\theta$.  Note that $x_i$ must be the rightmost letter in $\bu_n$  not to be assigned the value $1$ by $\theta$.  There are two occurrences of $x_i$ in $\bw_n$ and as the only other occurrence of $z_1$ in $\bw_m$ occurs left of the first occurrence of $x$, so too must the first occurrence of $x_i$ in $\bw_n$ be before the first occurrence of $x$.  Hence $i=0$.  However, as $x_i$ was the rightmost letter in $\bu_n$ not to be assigned the value $1$ (and as $n>1$) we must have $x_j\theta=1$ for all $0<j\leq n$ and $\bu_n\theta=x_0\theta$.  But this is not possible because $\bu_m\leq \bu_n\theta$, and $\bu_m$ does not occur twice in $\bw_m$.  

Hence we have $z_1xyz_2\theta= z_1xyz_2$.  We are now at the situation encountered in \cite{jaclee}.  We sketch the rest of the proof.  For each $j\leq m$, let $i_j$ be such that $x_j\leq x_{i_j}\theta$.  As $x_0$ appears either side of $z_1xyz_2$, and $x_{i_0}\theta$ appears twice in $\bw_n\theta$, it follows that $x_{i_0}$ appears either side of $z_1xyz_2$ in $\bw_n$.  So $i_0=0$.  Then $x_{i_1}$ must have an occurrence before the second occurrence of $x_0$ in $\bw_n$, so that $i_1=1$ also.  Then $x_{i_2}$ must have an occurrence after the second occurrence of $x_0$ and before the second occurrence of $x_1$ so that $i_2=2$ as well.  We can continue in this way until we uncover $i_m=m$ (and it is not possible to satisfy this matching if $n<m$).  Then the second occurrence of $x_m$ in $\bw_n$ must be after the second occurrence of $z_1xyz_2$ in $\bw_n$.  This forces $n=m$, as required.
\end{proof}
\subsection{Step B}
We now identify a series of isoterms for the identities in $\Sigma$.  To start with, observe that the identity system $\{xt_1xt_2x\approx x^3t_1t_2t_3\approx t_1t_2x^3,x^3\approx x^4\}$ is a basis for the variety gen\-er\-ated by $\bM(\{abba,aabb,abab\})$ (see \cite{JS}) while $\Sigma$ is easily seen to be satisfied by $\bM(\{abba,aabb\})$ (as no pair of letters $a,b$ in $\bw_n$ has $\bw_n[a,b]\in\{abba,aabb\}$); so the variety defined by $\Sigma$ lies somewhere between $\mathcal{V}\{\bM(\{abba,aabb\})\}$ and $\mathcal{V}\{\bM(\{abba,abab,aabb\})\}$.  Intuitively, the identities $\bw_n\approx \bw_n'$ appear quite weak, as they allow only one small switch in order of appearance of letters, and even then only when linked variables appear within a very specific pattern.
 The pattern of occurrences
 \[
 x_0xx_1x_0x_2x_1x_3x_2\dots x_nx_{n-1}xx_n
 \]
  is already easily seen to be an isoterm for  $\bM(\{abba,aabb\})$.
 We add new words to the pair $\{abba,aabb\}$ with the following strategy: first we add words to restrict the possible kinds of unstable pairs.  Then we add words that allow these unstable pairs to be linked \textit{only} with variables that appear in the pattern of occurrences that correspond to images of $\bw_n$ under linear substitutions.  (Here by a \textit{linear substitution} we mean one that assigns letters to $1$-limited words. The empty word is $1$-limited and is a valid substitution in the language of monoids.)  

\begin{lem}\label{lem:easy}
Let $\bw$ be a $2$-limited word and $\{x,y\}$ an unstable pair in an identity $\bw\approx \bw'$ satisfied by $\bM(\{xyyx,xxyy\})$.  Then $\bw[x,y]\approx \bw'[x,y]$ is the identity $xyxy\approx yxyx$ or $yxyx\approx xyxy$.
\end{lem}
\begin{proof}
First note that $xx$ is an isoterm, as are $xy$, $xxy$, $xyx$ and $yxx$ (all from the single word $xyyx$).  So the number of occurrences of $x$ and of $y$ must be exactly~$2$.  As $xyyx, yxxy,xxyy,yyxx$ are isoterms, the last remaining possible nontrivial identity is $xyxy\approx yxyx$ (or $yxyx\approx xyxy$).
\end{proof}

Let $\sU$ denote the set of words
\[
\{xyyx,xxyy,xtyxy,xytxy,xyxty,xyzyxz,
zxyzyx
\}.
\]
\begin{lem}\label{lem:Uisoterms}
The following are isoterms for $\bM(\sU)$\up:
\[
xyzxzy, 
xyxzzy, 
xyxzyz.
\]
\end{lem}
\begin{proof}
The first word is identical  up to a change of
letter names to the word $zxyzyx\in \sU$.  For the second word
$\bw:=xyxzzy$, observe that $\{x,z\}$ is stable (deletion of $y$
yields the isoterm $xxzz$) and $\{y,z\}$ is stable (deletion of $x$
yields $yzzy$).  Thus the only possible unstable pair is $\{x,y\}$.
But by Lemma~\ref{lem:easy}, if $xyxzzy\approx \bw'$ were a
nontrivial identity satisfied by $\bM(\sU)$ then $\bw'[x,y]=yxyx$.  This
is impossible because the first occurrence of $z$ must come before
the second occurrence of $y$ and after the second occurrence of $x$.
So $\{x,y\}$ is stable also.

Finally, for $\bw:=xyxzyz$ observe from Lemma~\ref{lem:easy} that
$\{x,z\}$ is stable, so that if $xyxzyz\approx \bw'$ is a satisfied
identity, then $\bw'[x,z]=xxzz$.  By Lemma~\ref{lem:easy}, if
$\{x,y\}$ is unstable, then  $\bw'[x,y]=yxyx$, making
$\bw'[y,z]=yyzz$, an isoterm, and contradicting the assumption
that $\bw\approx \bw'$ is satisfied.  Similarly if $\{y,z\}$ is unstable,
then Lemma~\ref{lem:easy} gives $\bw'[y,z]=zyzy$, giving
$\bw'[x,y]=xxyy$, also an isoterm and a contradiction.  Thus all
pairs of letters are stable and $xyxzyz$ is an isoterm.
\end{proof}
\begin{lem}\label{lem:adjacent}
Let $\bu\approx \bv$ be a $2$-limited identity satisfied by $\bM(\sU)$ and let $(x,y)$ be a critical pair in $\bu$.  Then $\bu$ can be written as $\bu_1xy\bu_2xy\bu_3$ \up(or reverse\up) for some possibly empty words $\bu_1,\bu_2,\bu_3$ in which $\bu_2$ contains no letters that are linear in $\bu$.
\end{lem}
\begin{proof}
As $(x,y)$ is unstable, Lemma \ref{lem:easy} shows that we may assume without loss of generality that $\bu[x,y]\approx \bv[x,y]$ is the identity $xyxy\approx yxyx$.  Because $(x,y)$ is critical, there are $i,j\leq 2$ with $_ix\, {}_jy$ occurring as a subword of $\bu$ but with $_ix$ occurring after $_jy$ in $\bv$.  This forces either $i=j=1$ or $i=j=2$.  Without loss of generality we may assume that $i=j=1$ giving $\bu= \bu_1xy\bu_2x\bu_3'y\bu_3$.  We show that $\bu_3'$ is empty and that $\bu_2$ has no letters that are linear in $\bu$.  The second statement follows immediately from the fact that $xytxy$ is an isoterm.  So now assume for contradiction that $\bu_3'$ is not empty: containing an occurrence of some letter $z$ say.  As $xyxty$ is an isoterm we have that $z$ is not linear in $\bu$.  So---given $_1x\, {}_1y$ is a subword of $\bu$---the possible pattern of occurrences of $z$ are as follows:
\[
zxyxzy, xyzxzy, xyxzzy, xyxzyz.
\]
The first word is identical up to a change of letter names to the isoterm $xyzyxz\in \sU$.  The remaining three words in the list are isoterms by Lemma \ref{lem:Uisoterms}.  Thus in each case we obtain a contradiction with the instability of
$(x,y)$.  
\end{proof}

Now we must add new words to $\sU$ that will restrict the pattern of
occurrence of other variables to very specific forms: those that can
be obtained from $\bw_n$ by linear substitutions.  The goal is to
build on Lemma~\ref{lem:adjacent} by ensuring that all letters
linked to $x$ (equivalently, $y$) will be covered by some linear
substitution from $\bw_n$.

The family of words we add will be obtained by interleaving
relatively short patterns amongst $xyxy$.  We will use
Lemma~\ref{lem:adjacent} to simplify to a case where the pattern of
occurrences of the two variables $x$ and $y$ is of the form
$\bu_1xy\bu_2xy\bu_3$.  Because of this, \textit{we only need to
consider interleavings that  keep $_ix\,{}_iy$ a subword for
$i=1,2$}.  

The following list of required isoterms might appear quite daunting,
and the reader might initially be doubtful that any careful check
can be made that the list is complete relative to allowing only
$\Sigma$ and its consequences.  We stress here that no check of
completeness is required at this stage.  All that needs to be
checked is that each given word is in fact an isoterm with respect
to $\Sigma$.  The groupings 1--9 relate to the pattern of letters
that are being interleaved around $xy\dots xy$, and is useful for
reference during the main proofs.  Similarly the choice of letter
names is arbitrary: we use $x,y$ consistently, but the order of
appearance of letters $a,b,c,d,e$ is chosen to match (where
possible) the order that certain letters appear in the main proof:
this makes it easier to check that certain patterns are isoterms.  
\begin{enumerate}
\item 
$atxyaxy$, and reverse,
\item 
$abaxybxy$, $abxyaxyb$, and their reverses,
\item 
$abbxyaxy$, $abxybaxy$,
$xyabbaxy$, and their reverses,

\item 
$abxyacbcxy$, 
\item 
$axybcbacxy$, $axybbcacxy$, $axybbaccxy$, $axybbacxyc$, and their reverses,
\item 
$axybcabxyc$,
\item $axybcabdcxyd$, 
$abxyacbddxyc$, 
$abxyacbdxycd$,
$axybcabdcdxy$, and reverses
\item $xyabcadbecdexy$,
\item
$xydcdabcaebexy$, $dxycdabcaebexy$, $dxycdabcaebxye$, and their reverses.
\end{enumerate}
For reference purposes we use the notation ($n$.$i$) to denote the $i$th word of kind $n$.  For example, isoterm (3.2) is the word $abxybaxy$.  
The words 3.3, 6.1, 7.3, 8.1, 9.1, 9.3, are identical up to a change of letter names to their own reverse and so there are really only $2\times 20-6=34$ words in the list.  We also need the words in $\sU$, however all of these except  $xxyy,xyzyxz,
zxyzyx$ occur as subwords of the new words of kinds 1--9.  We let  $\sV$ denote the $37$ word set consisting of $xxyy,xyzyxz,
zxyzyx$ along with the 34 words just listed.

Before verifying that these words are isoterms for $\Sigma$ it is convenient to introduce some further concepts targeted at analysing $2$-limited words.

 A \textit{block} of a word $\bw$ will be a maximal subword $\bu\leq \bw$ such that for all variables $x$, if $x\in \con(\bu)$, then all occurrences of $x$ in $\bw$ lie within $\bu$.  Note that the intersection of two blocks is again a block, so that for any subword $\bu\leq \bw$ there is a smallest block containing $\bu$.  Two letters $x,y$ will be said to \textit{interlock} in $\bw$ if $\bw[x,y]\in\{xyxy,yxyx\}$.  
 A letter $y$ is said to be \textit{linked} to $x$ if there is a sequence of letters $x=x_0,x_1,\dots,x_n$ such that for each $0\leq i<n$ the letter $x_i$ is interlocked with $x_{i+1}$ and either $x_n$ is interlocked to $y$ or $\bw[x_n,y]=x_nyyx_n$ or $\bw[x_n,y]=x_nyx_n$.  
 Note that in the case that $x_n$ is interlocked to $y$, then $x$ is linked to $y$ also.  Also note that when $n=0$ this means that $\bw[x,y]=xyyx$ (in this case $y$ is linked to $x$, but $x$ is not linked to $y$).  It is routine to verify that the smallest block containing a subword $\bu\leq \bw$ is the subword consisting of all letters that are linked to letters in $\bu$.  Note that  if $\bv$ and $\bw$ are $2$-limited words without linear letters and $\bw$ is itself a block, then $\bv\theta\leq \bw$ implies $\bv\theta=\bw$ for every substitution $\theta$.
 
\begin{lem}\label{lem:V}
All words in $\sV$ are isoterms with respect to $\Sigma$, so that $\bM(\sV)\models \Sigma$.
\end{lem}
\begin{proof} (Sketch.)
We have already argued this in the case of the words in $\sU$.  As all words in $V$ are $2$-limited, only identities from $\Sigma$ of the form $\bw_n\approx\bw_n'$ can possibly apply nontrivially to them.  
Next, a quick inspection of the words in $\sV\backslash \sU$ reveals that in every case the only subwords that appear more than once in any given word are individual letters and $xy$.  
Let $\bv$ be any one of the 37 words listed above, and assume for contradiction that there is $n$ and a substitution $\theta$ such that $\bw_n\theta\leq\bv$ and $\bw_n\theta\neq \bw_n'\theta$.  So neither $x\theta$ nor $y\theta$ is assigned the value $1$.  
But then $xy\theta$ is a subword of length more than $1$ and with two occurrences, and as $\bv$ contains exactly one such subword (namely $xy$), we have $xy\theta=xy$, which gives $x\theta=x$ and $y\theta$.  
As all other letters in $\bw_n$ are also $2$-occurring in $\bw_n$, this also gives $|x_i\theta|\leq 1$ and $|z_i\theta|\leq 1$.  Thus, $\bw_n\theta$ is obtained by deleting letters from $\bw_n$, and possibly changing the name of some letters (but not $x$ and $y$).  It is very easy (in each individual case) to see that this is not true of $\bv$.  
This task is easier to see by inspection than to write carefully, so we restrict to one ostensibly ``difficult'' case, and leave the rest for the reader to convince themselves in the other cases.

Let $\bv=dxycdabcaebxye$ (isoterm 9.3) and assume for contradiction that some subword of this is obtained from $\bw_n$ by a substitution $\theta$ that fixes $x$, $y$ and either deletes or renames all others.   As the smallest block containing $x$ is all of $\bv$ (a feature also shared by all other choices of $\bv\in \sV\backslash \sU$), it follows that $\bw_n\theta=\bv$.  Then $\theta(x_0z_1)=d$, but as $d$ does not appear immediately left of $_2x$ we cannot have $z_1\theta=d$.  So $x_0\theta=d$.  Similarly, we cannot have $c$ a prefix of $z_2\theta$ as $c$ does not have an occurrence after $_2y$.  So $z_2\theta=1$ and we must have $x_1\theta=c$ (the only letter between $_1z_2$ and $_2x_0$).  Then $x_2\theta=da$ (as $x_2$ is the only letter between $_2x_0$ and $_2x_1$).  This is a contradiction, as $da$ occurs just once in $\bv$ but $x_2$ occurs twice in $\bw_n$.
\end{proof}

\subsection{Step C}
Lemma shows that $\Sigma$ is irredundant, and Lemma \ref{lem:V} shows that $\Sigma$ is satisfied by $\bM(\sV)$.  We now show that every identity satisfied by $\bM(\sV)$ follows from $\Sigma$, which completes the proof of the main result.
\begin{thm}\label{thm:irredundant}
$\Sigma$ is an  identity basis for the identities of $\bM(\sV)$.
\end{thm}
\begin{proof}
This proof is more involved than the previous ones, so we outline the overall structure.  Our goal is to show that an arbitrary identity $\bu\approx \bv$ of $\bM(\sV)$ can be proved from $\Sigma$.  First we make a series of simplifications to the possible identities $\bu\approx \bv$ we need to consider.  At that point it will suffice to show how to us $\Sigma$ to remove a single critical pair from the set of unstable pairs of $\bu\approx \bv$.  We then show that, subject to the simplifications, $\bu$ is identical to $\bw_n\phi$ for some $n$ and substitution $\phi$.  The substitution $\phi$ is constructed inductively.

Let $\bu\approx \bv$ be a nontrivial identity of $\bM(\sV)$.  As $xx$ is an isoterm, any letter that appears in $\bu$ more than twice also appears in $\bv$ more than twice. Using the identities $xt_1xt_2x\approx xxxt_1t_2\approx t_1t_2xxx\approx  t_1t_2xxxx$ we can move all occurrences of letters with more than $2$ occurrences to the left to derive $\bu\approx z_1^3\dots z_k^3\bu'$ and $\bv\approx z_1^3\dots z_k^3\bv'$ where $\{z_1,\dots,z_k\}=\{z\in \con(\bu)\mid \occ_z(\bu)>2\}$ and $\bu'\approx \bv'$ is $2$-limited.  Thus there is no loss of generality in assuming from the start that both $\bu$ and $\bv$ are $2$-limited and balanced.

Let $(x,y)$ be a critical pair in $\bu\approx \bv$.  By
Lemma~\ref{lem:adjacent}, we may assume that $\bu=
\bu_1xy\bu_2xy\bu_3$ and $\bv= \bv_1y\bv_2x\bv_3y\bv_4x\bv_5$
(with $\bu_1,\bu_2,\bu_3,\bv_1,\bv_2,\bv_3,\bv_4,\bv_5$ possibly
empty and with no letter appearing in $\bu_2\bv_2\bv_3\bv_4$ being
linear in $\bu$).  We are going to show that $\Sigma$ can be applied to obtain
$\bu\approx \bu'$ such that $\bu'\approx \bv$ has one fewer unstable
pair: we remove the pair $(x,y)$.  More precisely, we shows that
$\Sigma\vdash \bu_1xy\bu_2xy\bu_3\approx \bu_1yx\bu_2yx\bu_3$.  Once
this is achieved, the proof will be complete: the number of unstable
pairs has been reduced by $1$, and as there are only finitely many
unstable pairs, repeating the process eventually provides a complete
proof of $\bu\approx \bv$.

We now make a series of simplifications, which give us greater control over the form of $\bu$.

{\bf Reduction 1.} \textit{It suffices to assume that $\bu$ is the
smallest block containing $x$.}  Equivalently: all letters in $\bu$
are linked to $x$.  Indeed, if there was some smaller block $\bu'$
of $\bu$ containing $x$ and we are able to apply $\Sigma$ to remove
the critical pair $(x,y)$ from $\bu'$ without adding further
unstable pairs, then the critical pair $(x,y)$ was also removed from
$\bu$ without adding further unstable pairs.

We are now going to work toward the construction of a substitution $\phi$ that has $\bw_n\phi=\bu$, and with $x\phi=x$, $y\phi=y$.  Then $\bw_n'\phi$ differs from $\bu$ only in the order of appearance of $x$ and $y$, as required.  The number $n$ will be determined later (only in particular cases is the value of $n$ unique).

{\bf Reduction 2}.  \textit{It suffices to assume that the final
letter of $\bu_1$ is not the final letter of $\bu_2$, and the first
letter of $\bu_2$ is not the first letter of $\bu_3$.  This will
force $z_1\phi=z_2\phi=1$.} To establish this, let $\bu_{1,2}$
be the largest suffix of $\bu_1$ that is also a suffix of $\bu_2$
(write $\bu= \bu_{1,1}\bu_{1,2}$), and let $\bu_{3,1}$ be the
largest prefix of $\bu_3$ that is also a prefix of $\bu_2$ (and
write $\bu_3= \bu_{3,1}\bu_{3,2}$).  Because $\bu$ is
$2$-limited, the prefix $\bu_{3,1}$ of $\bu_2$ and the suffix
$\bu_{1,2}$ of $\bu_2$ do not overlap.  So we may write $\bu=
\bu_{1,1}\bu_{1,2}xy\bu_{3,1}\bu_{2}'\bu_{1,2}xy\bu_{3,1}\bu_{3,2}$
for some word $\bu_2'$ whose first letter is distinct from the first
letter of $\bu_{3,2}$ and whose final letter is distinct from the
final letter of~$\bu_{1,1}$.  Moreover, because $\bu$ is $2$-limited
we have
\begin{multline*}
\con(\bu_{1,2})\cap\con(\bu_{1,1}xy\bu_{3,1}\bu_{2}'xy\bu_{3,1}\bu_{3,2})\\
=
\con(\bu_{3,1})\cap\con(\bu_{1,1}\bu_{1,2}xy\bu_{2}'\bu_{1,2}xy\bu_{3,2})=\varnothing.
\end{multline*}
So assigning $1$ to all letters in $\con(\bu_{1,2})\cup\con(\bu_{3,1})$ and leaving other letters unchanged will produce $\bu':=\bu_{1,1}xy\bu_{2}'xy\bu_{3,2}$, which is a word satisfying the assumptions being asked of $\bu$ in Reduction 2.  If we manage to find $n$ and $\phi$ with $x\phi=x$, $y\phi=y$ and $\bw_n\phi=\bu'$, then we must have $z_1\phi=z_2\phi=1$ because the final letter of $\bu_{1,1}$ is distinct from the final letter of $\bu_2'$ and the first letter of $\bu_2'$ is distinct from the first letter of $\bu_{3,2}$.  Let $\phi'$ be the substitution that agrees with $\phi$ on all letters except for $z_1$ and $z_2$, and has $z_1\phi'=\bu_{1,2}$ and $z_2\phi'=\bu_{3,1}$.  Then $\bw_n\phi'=\bu$ (with $x\phi'=x$ and $y\phi'=y$) as required.

{\bf Reduction 3.} It suffices to consider the case where the word
$\bu$ has the following property: \textit{the only nonempty subwords
of $\bu$ that occur twice in $\bu$ are individual letters and the
subword $xy$}.  To see why we can make this simplification, assume
that $\bu$ does not have this property.  Consider any maximal
subword $\bu'$ of $\bu$ that appears twice in $\bu$ but is not $xy$
itself.  Because of Reduction 2, $\bu'$ does not contain $x$ and
does not contain $y$.  Because $\bu$ is $2$-limited, the two
occurrences of $\bu'$ do not overlap, and moreover, if $\bu''$ is
some other maximal subword of $\bu$ occurring twice in $\bu$, then
$\bu''$ does not overlap with $\bu'$.  For each such maximal subword
$\bu'$ appearing twice in $\bu$, let $x_{\bu'}$ denote the first
letter of $\bu'$.  Note that if $\bu'$ is a single letter, then it
is just the letter $x_{u'}$.  Let $\bar{\bu}$ be the result of
replacing all maximal $2$-occurring subwords $\bu'$ of $\bu$ (except
for $xy$) by $x_{\bu'}$.   So $\bar{\bu}$ is of the form
$\bar{\bu}_1xy\bar{\bu}_2xy\bar{\bu}_3$, where
$\bar{\bu}_1,\bar{\bu}_2, \bar{\bu}_3$ are subject to the same
assumptions as $\bu_1,\bu_2,\bu_3$ and $\bar{\bu}$ has the property
in the statement of Reduction 3.  If we can reverse the order of
$xy$ in $\bar{\bu}$ then we are done, because
$\bar{\bu}_1xy\bar{\bu}_2xy\bar{\bu}_3\approx
\bar{\bu}_1xy\bar{\bu}_2xy\bar{\bu}_3\vdash
\bu_1xy\bu_2xy\bu_3\approx \bu_1yx\bu_2yx\bu_3$ using the
substitution that fixes $x$ and $y$ but has $x_{\bu'}\mapsto \bu'$.

We now continue our search for $\phi$ (with $z_1\phi=z_2\phi=1$, as required by Reduction 2).  We proceed by establishing a number of claims.  In each case, the proof is by contradiction, showing that if the claim were false, then $\{x,y\}$ could not be unstable in $\bu$.

\begin{claim}\label{claim1} If $a$ interlocks with $x$ and $b$ interlocks with $a$, then $b$ interlocks with $x$ also.
\end{claim}
\begin{proof}[Proof of Claim~\ref{claim1}.]
Assume that $a$ interlocks with $x$ and $b$ interlocks with $a$ but not~$x$.  So $\bu[a,b,x,y]\in\{babxyaxy,abxyaxyb,xyaxybab,bxyaxyba\}$.  But, up to a change of letter names these are isoterms (2.1), (2.2) and their reverses (respectively), which would contradict the instability of $\{x,y\}$ in $\bu$.
\end{proof}

\begin{claim}\label{claim2} Every letter in $\con(\bu)\backslash\{x,y\}$ is $2$-occurring and has at least one of its occurrences in $\bu_2$.
\end{claim}
\begin{proof}[Proof of Claim~\ref{claim2}.]
By Reduction 1, all letters in $\bu$ are linked to $x$.  Next observe that no letter in $\bu_2$ can be linear in $\bu$ because $xytxy$ is an isoterm (for example, $xyaxy$ is a subword of  isoterm (1.1)).  Now assume that $a$ is a letter occurring in $\bu_1$.  As $a$ is linked to $x$ there is a chain of successively interlocked letters leading from $x$ to $a$.  By Claim~\ref{claim1} we have that either $a$ itself is interlocked with $x$ or that $a$ is not interlocked with $x$ but is linked to $x$ via a single letter $b$ interlocked with $x$.  In the first instance $a$ has an occurrence in $\bu_2$ and is $2$-occurring as required.  In the second instance we have that $\bu[a,b,x,y]\in\{baxybxy, baaxybxy\}$ or reverse (the arrangements $abaxybxy$ and $baxybxya$ are impossible due to Claim~\ref{claim1}).  However, $baxybxy$ is  isoterm (1.1), while $baaxybxy$ is isoterm (3.1); these would contradict the instability of $\{x,y\}$ in $\bu$.
\end{proof}

\begin{claim}\label{claim3}
$|\bu_1|\leq 1$ and $|\bu_3|\leq 1$.
\end{claim}
\begin{proof}[Proof of Claim~\ref{claim3}.]
Up to symmetry we may consider just $|\bu_1|$.  Assume for contradiction that $|\bu_1|>1$, and let the $2$-letter suffix of $\bu_1$ be $ab$.  Now $a$ and $b$ are distinct letters, because of Claim~\ref{claim2} and the fact that $\bu$ is $2$-limited.  By Claim~\ref{claim2} again, the second occurrence of both $a$ and $b$ is in $\bu_2$.  Because $abxybaxy$ is  isoterm~(3.2), we must have $\bu[a,b,x,y]=abxyabxy$.  By Reduction~3, there is some letter $c$ that appears between ${{}_2}a$ and ${{}_2}b$; moreover we may choose~$c$ so that it has an occurrence immediately to the left of $_2b$.  Now, if $c$ has an occurrence in $\bu_1$, then as $ab$ is a suffix of $\bu_2$ we have $\bu[a,b,c,x,y]=cabxyacbxy$.  However, $caxyacxy$ is also isoterm~(3.2), a contradiction with the instability of $\{x,y\}$ in~$\bu$.  Thus we may assume that $c$ has no occurrence in $\bu_1$.  So $\bu[b,c,x,y]$ is one of $bxyccbxy$,  $bxycbcxy$ or $bxycbxyc$.  We treat each as a separate subcase.\\
{\bf Subcase 1. }  $\bu[b,c,x,y]=bxyccbxy$ leads to contradiction.  \\
By Reduction~3  (and the fact that $_1b$ is adjacent to $_1x$ in $\bu$) there is some further letter $d$ with an occurrence between $_2b$ and~$_2x$.  As $dbxybdxy$  and $xydccdxy$ are isoterms (3.2) and (3.3), respectively, and $c$ occurs in $\bu$ immediately left of $_2b$, the only viable placements of $d$ in $\bu[b,c,d,x,y]$ are
$bxycdcbdxy$,
$bxyccbddxy$ and
$bxyccbdxyd$, which are all equivalent up to a change of letter names to isoterms of kind $5$ (a contradiction):
\begin{itemize}
\item $bxycdcbdxy$ is equivalent to $axybcbacxy$ isoterm (5.1);
\item $bxyccbddxy$ is equivalent to $axybbaccxy$ isoterm (5.3);
\item $bxyccbdxyd$ is equivalent to $axybbacxyc$ isoterm (5.4).
\end{itemize}
{\bf Subcase 2.}
$\bu[b,c,x,y]=bxycbcxy$ leads to contradiction.\\
  Recall the placement of $a$ and obtain  $\bu[a,b,c,x,y]=abxyacbcxy$, which is  isoterm~(4.1) (a contradiction).\\
  {\bf Subcase 3.}
Finally, consider
$\bu[b,c,x,y]=bxycbxyc$.\\
Again, let $d$ be a letter appearing between $_2b$ and $_2x$, which exists by Reduction~3 and the fact that $_1b\, {}_1x$ is a subword of $\bu$.  Now $dbxybdxy$ and $xycdxydc$ are isoterm (3.2) and its reverse.  So, given that $_1c$ occurs immediately left of $_2b$ and~$_1b$ occurs immediately left of $_1x$, the possible placements for $d$ in $\bu[b,c,d,x,y]$ are $bxydcbdxyc, bxycbddxyc, bxycbdxycd$.  The first of these is  isoterm (6.1), a contradiction.  The second and third possibilities are not themselves isoterms relative to~$\Sigma$ (for example, if $x_0\mapsto b$, $x\mapsto x$, $y\mapsto yc$, $x_2\mapsto d$ and all other letters are assigned~$1$, then $\bw_3$ maps onto $bxycbddxyc$ and $\bw_3\approx \bw_3'$ yields $bxycbddxyc\approx bycxbddycx$).  However, they do turn out to be isoterms once the position of letter $a$ is recalled.  With $a$ reintroduced we have $abxyacbddxyc$ and $abxyacbdxycd$, which are isoterms~(7.2) and~(7.3) respectively.

Thus in every case, subcase 3 leads to contradiction, as required.  Thus all three of the possible cases  lead to the desired contradiction and the claim is established.
\end{proof}

Claim~\ref{claim3} shows that $\bu_1$ is either empty or a single letter.  In either case we will eventually be forced to select $\phi(x_0)=\bu_1$, and so to aid the notation, we now let $a_0$ (either a single letter or the empty word) denote $\bu_1$.  A dual argument from the right would identify a unique letter $a_n$ (or empty word) appearing after $_2y$.  The value of $n$ is yet to be determined.

\begin{claim}\label{claim4} If $a,b\in\con(\bu)\backslash\{x,y\}$, then $\bu(a,b)\notin\{abba, baab\}$.
\end{claim}
\begin{proof}[Proof of Claim~\ref{claim4}.]
Without loss of generality we may assume that $a$ has its first occurrence in $\bu$ before $b$, so that we need only ensure that $\bu[a,b]\neq abba$.  Assume for contradiction that $\bu[a,b]=abba$.  Claim~\ref{claim3} shows that at most one letter has an occurrence before $_1x$ and at most one letter has an occurrence after $_2y$.  On the other hand, $xyabbaxy$ is  isoterm (3.3), so at least one occurrence of $a$ is outside of the subword $\bu_2$.  Now $\bu[a,x,y]=axyxya$ is impossible as it violates Claim~\ref{claim2}.  Thus the only possibilities are $\bu[a,x,y]=axybbaxy$ or $\bu[a,x,y]=xyabbxya$.  By symmetry it suffices to show that  $\bu[a,b,x,y]=axybbaxy$ leads to contradiction.

By Claim~\ref{claim3}, $_1a\,{}_1x$ is a subword.  Then by Reduction~3, there must be some letter~$c$ occurring between $_2a$ and $_2x$.
The possible configurations for the placement of~$c$ are
$axycbbacxy$, $axybcbacxy$, $axybbcacxy$, $axybbaccxy$ and $axybbacxyc$.  The first of these deletes to $xycbbcxy$, which is isoterm (3.3).  The remaining four configurations are isoterms (5.1), (5.2), (5.3) and (5.4) respectively.  Thus all cases
result in a contradiction with the instability of $\{x,y\}$,
completing the proof of the claim.
\end{proof}

\begin{claim}\label{claim5} If $a_0$ is nonempty and there is a letter $b$ between $_1y$ and $_2a_0$, then the second occurrence of $b$ is to the right of the second occurrence of~$a_0$. Moreover, if this letter $b$ has an occurrence after $_2y$, then $\bu=a_0xyba_0xyb$.  \up(The right dual statement refers to $a_n$ and any letter appearing between $_1a_n$ and $_2x$.\up)
\end{claim}
\begin{proof}[Proof Claim~\ref{claim5}.]
The first statement is an immediate consequence of Claim~\ref{claim3} and Claim~\ref{claim4}.  For the second statement, first let us use the notation $a$ for $a_0$.  Now consider the situation where $_2b$ occurs after $_2y$.  So $\bu[a,b,x,y]=axybaxyb$.  Now, by Claim~\ref{claim3}, if there is a further letter in $\bu$, it must have both of its occurrences between $_1y$ and $_2x$.  The only placement of this letter, $c$, that does not violate Claim~3 or Claim~\ref{claim4} is $axycbacxyb$.  But this is equivalent to isoterm~(6.1).   Thus no further letters exist, and the proof is complete.
\end{proof}

It is not clear at this point of the proof that if there is a letter $b$ satisfying the assumptions of Claim~\ref{claim5}, then it is unique.  In fact it will be unique, but at this stage it is sufficient to let $a_1$ denote the choice of $b$ in Claim~\ref{claim5} whose first occurrence immediately follows $_1y$ (if such a letter exists; otherwise $a_1$ can be the empty word).  We reuse the letter $b$ for other purposes during proofs below.

\begin{claim}\label{claim6}
If $a_0$ and $a_1$ are nonempty, then $a_0xya_1a_0$ is a prefix of $\bu$.  \up(The right dual statement states that if $a_{n}$ and $a_{n-1}$ are nonempty then $a_na_{n-1}xya_n$ is a suffix of $\bu$.\up)
\end{claim}
 \begin{proof}[Proof of Claim~\ref{claim6}.] For the duration of the proof we rename $a_0$ and $a$ and $a_1$ as~$b$.    Assume for contradiction that $c$ is some third letter adjacent (to the right) of $_1b$ and before $_2a$.  If there is a choice, then we select $c$ to have an occurrence immediately right of $_1b$.  By Claim~\ref{claim5}, the second occurrence of $c$ is to the right of $_2a$.  Also, as $b$ and $c$ are distinct, the final statement of Claim~\ref{claim5} shows that the second occurrence of both $b$ and $c$ are left of the second occurrence of $x$.  As $xybccbxy$ is isoterm (3.3), the second occurrence of $c$ is also to the right of $_2b$.  So we have  $\bu[a,b,c,x,y]=axybcabcxy$.  However, as $_1b\, {}_1c$ is a subword, Reduction~3 implies that there is a letter $d$ with an occurrence between $_2b$ and $_2c$.  If $d$ has an occurrence to the right of $_2y$, then we have $\bu[a,b,c,d,x,y]= axybcabdcxyd$.  This is isoterm (7.1), a contradiction.  If $d$ has an occurrence after $_2c$ but before $_2x$ then $\bu[a,b,c,d,x,y]= axybcabdcdxy$, which is isoterm (7.4).   So, by Claim~\ref{claim4} we must search for the first occurrence of $d$ left of $_1c$.  But $_1c$ is adjacent to $_1b$, which is adjacent to $_1y$, so no placement is possible without contradicting the fact that $a$ is the only letter left of $_1x$ (by Claim~2).  Thus $b$ (if it exists) is the only letter with an occurrence between $_1y$ and $_2a$.  This completes the proof of Claim~\ref{claim6}.
\end{proof}

By symmetry, we may likewise work from the right and identify a unique (or empty) letter $a_n$ occurring after $_2y$, and the unique (or empty) letter $a_{n-1}$ whose second occurrence occurs between $_1a_n$ and $_2x$.  The precise value of $n$ is determined later.  The case where $n=1$ has already been encountered and completed: it coincides with the situation where $_2a_1$ occurs to the right of $_2y$, which is completed  by the second statement of Claim~\ref{claim5}.

Now we consider the case where $a_0$ is empty. If $\bu_2$ is empty also then $\bu=xyxy$ and we are done.  Otherwise $\bu_2$ is nonempty and we let $a_1$ denote the leftmost letter in $\bu_2$.

The arguments up to now may be considered the base case in an inductive construction of $\phi$ (with $x\phi=x$, $y\phi=y$ and $z_1\phi=z_2\phi=1$).  For the general {\bf Induction Hypothesis}, assume that for some $k\geq 1$ we have identified $a_0,\dots,a_k$ and defined $x_j\phi=a_j$ for each $j=0,\dots,k$ such that
\begin{itemize}
\item[(i)] each of $a_0,\dots,a_k$ is either a letter in $\con(\bu)$ or is empty and
\item[(ii)] $x_0xyx_1x_0x_2x_1\dots x_{k-1}x_{k-2}x_{k}x_{k-1}\phi$ is a prefix of $\bu$, and
\item[(iii)] there is no $j<k$ such that both $a_j$ and $a_{j+1}$ are empty.
\end{itemize}
Let $\bu'$ denote the prefix $x_0xyx_1x_0x_2x_1\dots x_{k-1}x_{k-2}x_{k}x_{k-1}\phi$ and $\bu''$ be such that $\bu=\bu'\bu''$.  The following observations are consequences of Induction Hypothesis~(ii) and the fact that each letter with an occurrence in $\bu$ has exactly two occurrences in $\bu$: each of $a_0,\dots,a_{k-1}$ has no  occurrences in $\bu''$; if nonempty, the letter $a_k$ has precisely one  occurrence in $\bu''$; any letter in $\con(\bu'')\backslash\{x,y,a_k\}$ has no occurrences in $\bu'$.  We use these frequently throughout the remainder of the argument.

In order to define $a_{k+1}$ there are three cases to consider according to whether none or exactly one  of $a_k,a_{k+1}$ are empty (by Induction Hypothesis (iii) it is not possible that both are nonempty).  It is useful to first establish the following facts, which explain the role of the most technical isoterms of kinds 8 and 9.

\begin{claim}\label{claim7}
Let $a,b,c$ be distinct letters in $\con(\bu)\backslash \{x,y\}$.  Then neither $_1a\,{}_1b\,{}_1c$ nor $_2a\,{}_2b\,{}_2c$ are subwords of $\bu$.
\end{claim}
\begin{proof}[Proof of Claim~\ref{claim7}]
By symmetry we may assume for contradiction that $_1a\,{}_1b\,{}_1c$
is a subword of $\bu$.  By Claim~\ref{claim3} we have
$_1a\,{}_1b\,{}_1c$ a subword of $\bu_2$.  Also, by
Claim~\ref{claim4} we must have $\bu[a,b,c]=abcabc$.  By
Claim~\ref{claim3} again, if any of $a,b,c$ have a second occurrence
right of $_2y$, then it is $c$ only.  However,  this violates the
right dual to Claim~\ref{claim6}, which would have $_1c$ adjacent to
$_2b$.  Hence each of $a,b,c$ have both occurrences within $\bu_2$.
Now, by Reduction~3, there must a be a letter $d\notin\{x,y\}$ having
an occurrence between $_2a$ and $_2b$.  As each of $a,b,c$ is
$2$-occurring, we have $d\notin \{a,b,c\}$.  By Claim~\ref{claim4}
we have $\bu[d,c]\neq cddc$, so that either the first occurrence of
$d$ is left of $_1c$ or the second occurrence of $d$ is right of
$_2c$.  As $_1a\,{}_1b\,{}_1c$ is a subword of $\bu$, if the first
occurrence of $d$ was left of $_1c$ then it is also left of $_1a$
giving $\bu[a,d]=daad$ which contradicts Claim~\ref{claim4}.  Thus
$\bu[a,b,c,d]=abcadbcd$.

Similarly by Reduction 3 there must be a letter $e\notin\{x,y\}$
having an occurrence between $_2b$ and $_2c$ (and as each of
$a,b,c,d$ is $2$-occurring, we have $e\notin \{a,b,c,d\}$).  Again
Claim~\ref{claim4} gives the second occurrence of $e$ right of $d$
so that $\bu[a,b,c,d,e]=abcadbecde$.  By Claim~\ref{claim3}, only
$e$ can possibly have an occurrence after $_2y$.  If $_2e$ is left
of $_2x$ then $\bu[a,b,c,d,e,x,y]=xyabcadbecdexy$, which is  isoterm (8.1), a contradiction.  If $_2e$ is $a_n$ (that is, right
of $_2y$), then $\bu[c,d,e,x,y]=xycdecdxye$, contradicting the right
dual of Claim~\ref{claim6}.
\end{proof}

\begin{claim}\label{claim8}
Let $a,b,c$ be distinct letters in $\con(\bu)$ with $a,b\notin \{x,y\}$.  If $_1a\,{}_1bc$ is a subword, then $c=a$.  Dually, if $b,c\notin\{x,y\}$ and $a\, {}_2b\,{}_2c$ is a subword of $\bu$ then $a=c$.
\end{claim}
\begin{proof}[Proof of Claim~\ref{claim7}]
Up to symmetry it suffices to assume for contradiction that
$_1a\,{}_1bc$ is a subword of $\bu$ (where $a,b\notin\{x,y\}$).  We
will assume that this is the leftmost instance of such a pattern in
$\bu$.  If $c=a$ we are done, so assume that $c\neq a$.  Trivially,
$c\neq y$.  By Claim~\ref{claim6}, we cannot have $c=x$.   If $c=b$
then $\bu[a,b]=abba$, contradicting Claim~\ref{claim4}.  Thus $c$ is
some letter in $\con(\bu)\backslash\{a,b,x,y\}$.  By Claim
\ref{claim4} we have $\bu[a,b]=abab$.  By Claim~\ref{claim7} the
given occurrence of $c$ is not the first occurrence of $c$, and as
$\bu[c,a]\neq caac$ and $\bu[c,b]\neq cbbc$ we have
$\bu[a,b,c]=cabcab$.  By Claims~\ref{claim5} and~\ref{claim6}, all
three of $a,b,c$ have both their occurrences in $\bu_2$.  By
Claim~\ref{claim7} and the fact that $_1a\,{}_1b$ is a subword, we
must have a letter $d$ with an occurrence $_id$ immediately
following $_1c$ and before $_1a$.  Now, the letter following $_id$
cannot be $c$ (as $_2c$ is right of $_1b$).  Then $i=2$, because
$a,b,c$ were chosen to be the leftmost instance of a failure of
Claim~\ref{claim8}.  So $\bu[a,b,c,d]=dcdabcab$.

By Reduction 3 and the fact that $_1a\,{}_1b$ is a subword, there must be some letter $e$ with an occurrence between $_2a$ and $_2b$.  By Claim~\ref{claim4}, if both occurrences of $e$ are left of $_2b$, then the first occurrence of $e$ must be left of $_1b$.  But $_1a\,{}_1b$ is a subword of $\bu$, so that the first occurrence of $e$ is also left of $_1a$, giving $\bu[a,e]=eaae$, contradicting Claim~\ref{claim4}.  Thus the second occurrence of $e$ is right of~$_2b$, giving $\bu[a,b,c,d,e]=dcdabcaebe$.  But then $\bu[a,b,c,d,e,x,y]$ is one of
\[
xydcdabcaebexy, dxycdabcaebexy, xydcdabcaebxye, dxycdabcaebxye.
\]
The first, second and fourth of these are isoterms (9.1), (9.2) and (9.3) respectively. The third word is the reverse of isoterm (9.2).
\end{proof}


%

Now we may present the inductive step.

\begin{indcase}\label{indcase1} Both $a_{k-1}$ and $a_k$ are nonempty. In this case $_1a_k\, _2a_{k-1}$ is a subword by inductive hypothesis \up(ii\up).  Let $c$ denote the letter immediately right of $_2a_{k-1}$.  If\up:
\begin{enumerate}
\item[(a)] $c=a_k$, then we define $a_{k+1}$ to be empty\up;
\item[(b)] $c=x$, then the second occurrence of $a_k$ is right of $_2y$ and we set $n=k$ and have $\bw_n\phi=\bu$, completing the inductive construction of $\phi$\up;
\item[(c)] $c\notin\{a_k,x\}$ then we define $a_{k+1}:=c$ and observe that  $_2a_{k-1}\,{}_1a_{k+1}\,{}_2a_{k}$ is a subword of $\bu$.
\end{enumerate}
\end{indcase}
\begin{proof}[Proof of Induction Case~\ref{indcase1}.]
 It is easy to see that the induction hypothesis is maintained, provided the stated observations are verified.
There is nothing to verify in case (a).  For case (b), it is trivial that the second occurrence of $a_k$ is right of $_2y$ (as $a_k$ must be $2$-occurring).  By the right hand dual to Claim~\ref{claim3}, there is at most one letter right of $_2y$, so that we now have $\bw_n\phi=\bu$, where $n:=k$.

Finally, we must verify the claim in (c) that
$_2a_{k-1}\,{}_1a_{k+1}\,{}_2a_{k}$ is a subword of $\bu$. For
notational simplicity, let us temporarily use $a$ to denote
$a_{k-1}$ and $b$ to denote $a_k$, continuing to use $c$ for
$a_{k+1}$, the letter with its first occurrence immediately right of
$_2a_{k-1}={}_2a$ by construction.   Assume for contradiction that
the letter immediately right of $_1c$ is $d\notin \{b\}$.  If $d=x$
then both $c$ and $d$ have their second occurrence after $_2y$,
contradicting the right dual to Claim~\ref{claim3}.  So we may
assume that $d\neq x$.  By Induction Hypothesis (ii), this is the
first occurrence of $d$.  But then by Claim~\ref{claim8}, the letter
following $_1d$ is either $c$ or $x$.  It cannot be $c$ as then
$\bu[b,c]=bccb$.  It cannot be $x$ by the right dual of Claim
\ref{claim6}.  Thus in every case we have a contradiction.  Thus
$_2a_{k-1}\,{}_1a_{k+1}\,{}_2a_{k}$ is a subword of $\bu$ as
required.
\end{proof}
%
%
%

\begin{indcase}\label{indcase2} $a_{k-1}$ is empty but $a_{k}$ is not.  Let $b$ be the letter immediately right of $_1a_k$.  If\up:
\begin{enumerate}
\item[(a)]  $b=a_{k}$ \up(so that $a_ka_k$ is a subword\up) then we define $a_{k+1}$ to be empty\up;
\item[(b)] $b=x$ then the second occurrence of $a_k$ is right of $_2y$ and we set $n=k$ and have $\bw_n\phi=\bu$\up;
\item[(c)] $b\notin\{a_k,x\}$ and we set $a_{k+1}:=b$.  In this case $a_{k}a_{k+1}a_{k}$ is a subword of $\bu_2$ and the second occurrence of $a_{k+1}$ is right of $_2a_k$.
\end{enumerate}
\end{indcase}
\begin{proof}[Proof of Induction Case~\ref{indcase2}.]
  Again, it is easy to see that the induction hypothesis is maintained, provided the stated observations are verified. There is nothing to verify in Case (a).  The argument for Case (b) is essentially identical to Case (b) of Induction Case~\ref{indcase1} and we do not repeat it here.

 For Case (c) note that $a_{k+1}$ cannot have an occurrence left of $_1a_k$ by Inductive Hypothesis (ii).  So it suffices to show that there is a single letter occurring between the two occurrences of $a_k$.  Assume for contradiction that this is not true. For the remainder of the proof we let $a:=a_k$, $b:=a_{k+1}$.   Let the letter immediately right of $_1b$ be denoted by $c$, so that $abc$ is a subword of $\bu$.  We have the subword $_1a\,{}_1b\,{}_ic$ for some $i\in\{1,2\}$.  By Claim~\ref{claim7}, we cannot have $i=1$; so $i=2$.  Then by Claim~\ref{claim8} we have $c=a$ as required.
\end{proof}

 \begin{indcase}\label{indcase3}
 $a_{k}$ is empty and $a_{k-1}$ is not empty.  In this case $a_{k-1}a_{k-1}$ is a subword of $\bu$ by Induction Hypothesis \up(ii\up).  Let $c$ be the next letter right of $_2a_{k-1}$.  If\up:
 \begin{enumerate}
 \item $c=x$ then we put $n=k$ and observe that $\bw_n\phi=\bu$ as required\up;
 \item $c\neq x$ then define $a_{k+1}:=c$.
 \end{enumerate}
 \end{indcase}
 \begin{proof}
  Again, it is easy to see that the induction hypothesis is maintained, provided the stated observations are verified.
 The only claim appears in (1).  This follows because if $c=x$ then $a_n$ is empty (by Claim~\ref{claim2}), so that $\bw_n\phi=\bu$ follows immediately.
 \end{proof}

This completes the inductive construction of the substitution
$\phi$.  As $\bu$ is finite, the value of $n$ is eventually
uncovered, and we have $\bw_n\phi=\bu$ (with $x\phi=x$,
$y\phi=y$) as required. This shows that we may remove the unstable
pair $\{x,y\}$ from $\bu\approx \bv$ using $\Sigma$ (specifically,
using $\bw_n\approx \bw_n'$).  Thus $\Sigma$ is a basis for the
equational theory of $\bM(\sV)$, completing the proof of
Theorem~\ref{thm:irredundant}. \end{proof}

\end{document}